\documentclass[11pt]{article}

\usepackage[margin=1.22in]{geometry}

\usepackage{amssymb,amsmath}
\usepackage{amsthm}
\usepackage{hyperref}
\usepackage{mathrsfs}
\usepackage{textcomp}
\hypersetup{
    colorlinks,%
    citecolor=black,%
    filecolor=black,%
    linkcolor=black,%
    urlcolor=black
}

\usepackage{verbatim}
\usepackage{graphicx}
\usepackage{sectsty}

\makeatletter

\newdimen\bibspace
\renewenvironment{thebibliography}[1]{%
 \section*{\refname 
       \@mkboth{\MakeUppercase\refname}{\MakeUppercase\refname}}%
     \list{\@biblabel{\@arabic\c@enumiv}}%
          {\settowidth\labelwidth{\@biblabel{#1}}%
           \leftmargin\labelwidth
           \advance\leftmargin\labelsep
           \itemsep\bibspace
           \parsep\z@skip     %
           \@openbib@code
           \usecounter{enumiv}%
           \let\p@enumiv\@empty
           \renewcommand\theenumiv{\@arabic\c@enumiv}}%
     \sloppy\clubpenalty4000\widowpenalty4000%
     \sfcode`\.\@m}
    {\def\@noitemerr
      {\@latex@warning{Empty `thebibliography' environment}}%
     \endlist}

\makeatother

\makeatletter

\newtheorem{thm}{Theorem}[section]

\newtheorem{prop}[thm]{Proposition}

\def\XXint#1#2#3{{\setbox0=\hbox{$#1{#2#3}{\int}$}
  \vcenter{\hbox{$#2#3$}}\kern-.5\wd0}}

                \newcommand{\lda}{\lambda}
\newcommand{\om}{\Omega}                \newcommand{\pa}{\partial}
\newcommand{\va}{\varepsilon}           \newcommand{\ud}{\mathrm{d}}
\newcommand{\be}{\begin{equation}}      \newcommand{\ee}{\end{equation}}

\newcommand{\R}{\mathbb{R}}

\begin{document}

\title{\textbf{On the sharp constants in the regional fractional Sobolev inequalities}
\bigskip}

\author{\medskip  Rupert L. Frank\footnote{R. L. F. was partially supported by the US National Science Foundation grant DMS-1954995 and the Deutsche Forschungsgemeinschaft (DFG, German Research Foundation) through Germany’s Excellence Strategy EXC-2111-390814868}, \quad Tianling Jin\footnote{T. J. was partially supported by NSFC 12122120 and Hong Kong RGC grant  GRF 16302519.}, \quad  Wei Wang}

\date{\today}

\maketitle

\begin{abstract}
In this paper, we study the sharp constants in fractional Sobolev inequalities associated with the regional fractional Laplacian in domains. 
\end{abstract}

\section{Introduction} 

Let $n\ge 1$, $\sigma\in (0,1)$ (with the additional assumption that $\sigma<1/2$ if $n=1$), and $\Omega\subset\R^n$ be an open set. Consider the sharp constant of the fractional Sobolev inequality
\[
Y_{n,\sigma}(\Omega):= \inf\left\{I_{n,\sigma,\R^n}[u]: u \in C^\infty_c(\Omega), \int_{\Omega} |u|^\frac{2n}{n-2\sigma}\,\ud x=1 \right\}\,,
\]
where
\begin{equation}\label{eq:sobolevnorm1}
I_{n,\sigma,\R^n}[u] := \iint_{\R^n\times\R^n} \frac{(u(x)-u(y))^2}{|x-y|^{n+2\sigma}} \,\ud x\,\ud y
\end{equation}
is the fractional Sobolev semi-norm of $u$. Using the dilation and translation invariance of $Y_{n,\sigma}(\R^n)$, it is not difficult to see that $
Y_{n,\sigma}(\Omega) = Y_{n,\sigma}(\R^n)$. Moreover, Lieb \cite{Lieb} classifies all the minimizers for $Y_{n,\sigma}(\R^n)$ and shows that they do not vanish anywhere on $\R^n$. Therefore, the infimum $Y_{n,\sigma}(\Omega)$ is not attained unless $\Omega=\R^n$.

Together with Xiong, the first two authors in \cite{FJX} considered the sharp constant of the fractional Sobolev inequality on the domain $\Omega$:
\begin{equation}\label{eq:sobolevconstant2}
S_{n,\sigma}(\Omega):= \inf\left\{I_{n,\sigma,\Omega}[u]: u \in C^\infty_c(\Omega),  \int_{\Omega} |u|^\frac{2n}{n-2\sigma}\,\ud x=1\right\},
\end{equation}
where 
\begin{equation}\label{eq:sobolevnorm2}
I_{n,\sigma,\Omega}[u] := \iint_{\Omega\times\Omega} \frac{(u(x)-u(y))^2}{|x-y|^{n+2\sigma}} \,\ud x\,\ud y
\end{equation}
is another fractional Sobolev semi-norm for $u$. In probability, $I_{n,\sigma,\Omega}$ defined in \eqref{eq:sobolevnorm2} is called the Dirichlet form of the censored $2\sigma$-stable process \cite{BBC} in $\Omega$. Its generator 
\begin{equation}\label{eq:regional}
(-\Delta)_{\Omega}^{\sigma}u:=2 \lim_{\va\to 0} \int_{\{y\in\Omega:\ |y-x|\ge \va\}} \frac{u(x)-u(y)}{|x-y|^{n+2\sigma}}\,\ud y
\end{equation}
is usually called the regional fractional Laplacian \cite{Guan, GuanMa}. Therefore, in this paper, we call \eqref{eq:sobolevnorm2} as the regional fractional Sobolev semi-norm of $u$. When $S_{n,\sigma}(\Omega)>0$, we call \eqref{eq:sobolevconstant2} the sharp constant of the regional fractional Sobolev inequality in $\Omega$. It follows from \cite{DyFr} that if $n\geq 2$ and $\sigma>1/2$, then $S_{n,\sigma}(\Omega)>0$. If $\sigma<1/2$ and $\Omega$ is of finite measure with sufficiently regular boundary, then Lemma 16 in \cite{FJX} shows that $S_{n,\sigma}(\Omega)=0$. If $\sigma<1/2$ and $\Omega$ is the complement of the closure of a bounded Lipschitz domain or a domain above the graph of a Lipschitz function, then it follows from the fractional Sobolev inequality on $\R^n$ and the Hardy inequality in \cite{Dy} that $S_{n,\sigma}(\Omega)>0$.  

It was discovered in \cite{FJX}  that  the minimization problem for $S_{n,\sigma}(\Omega)$ behaves differently from that for $Y_{n,\sigma}(\Omega)$. Unlike $Y_{n,\sigma}(\Omega)$, which always equals to $Y_{n,\sigma}(\R^n)$ and is never achieved unless $\Omega=\R^n$, the constant $S_{n,\sigma}(\Omega)$ depends on the domain $\Omega$, and can be achieved in many cases assuming that $n\ge 4\sigma$:
 
\begin{itemize}
\item If  the complement $\Omega^c$ has an interior point, then  
$
S_{n,\sigma}(\om)<S_{n,\sigma}(\R^n). 
$
\item If $\sigma\neq 1/2$, then $S_{n,\sigma}(\R^n_+)$ is achieved (see also Musina-Nazarov\cite{MN}).     
\item If $\sigma>1/2$, $\Omega$ is a bounded domain such that $B_\va^+\subset\Omega\subset \R^n_+$ for some $\va>0$, 
then $
S_{n,\sigma}(\om)<S_{n,\sigma}(\R^n_+).
$
Moreover, if $\partial \Omega $ is smooth then $S_{n,\sigma}(\om)$ is achieved.
\end{itemize}
Here, we used the notations that $\R^n_+=\{x=(x',x_n)\in\R^n: x_n>0\}$, $B_r=\{x\in\R^n: |x|<r\}$ and $B_r^+=B_r\cap\R^n_+$. 

Recently, Fall-Temgoua \cite{FT} proved that if $\Omega$ is a bounded $C^1$ domain and $\sigma>1/2$ is very close to $1/2$, then $ S_{n,\sigma}(\om)<S_{n,\sigma}(\R^n_+) $, and consequently,  $S_{n,\sigma}(\om)$ is achieved, by showing the upper semicontinuity of $S_{n,\sigma}(\om)$ on $\sigma\in[1/2,1)$ and using the fact that $S_{n,1/2}(\om)=0$.

As explained in \cite{FJX}, the discrepancy between the $S_{n,\sigma}(\Omega)$ and $Y_{n,\sigma}(\Omega)$ problems can be explained as a Br\'ezis-Nirenberg \cite{BrNi}  effect :
\begin{align*}
I_{n,\sigma,\Omega}[u] &= I_{n,\sigma,\R^n}[u] - 2\int_{\Omega}u^2(x)\,\ud x\int_{\R^n\setminus\Omega}\frac{1}{|x-y|^{n+2\sigma}} \,\ud y\\
& \approx I_{n,\sigma,\R^n}[u] - c_{n,\sigma}\int_{\Omega}\frac{u^2(x)}{\mbox{dist}(x,\partial\Omega)^{2\sigma}}\,\ud x\quad\forall\,u\in C_c^\infty(\Omega).
\end{align*}
Therefore, the $S_{n,\sigma}(\Omega)$ problem is the $Y_{n,\sigma}(\Omega)$ problem with an additional negative term, and  it is this term that for $n\ge4\sigma$ lowers the value of the infimum and produces a minimizer. This fact was first observed by  Br\'ezis-Nirenberg \cite{BrNi}. The difference between the $S_{n,\sigma}(\Omega)$ and $Y_{n,\sigma}(\Omega)$ problems is also related to the difference between the regional fractional Laplacian and the ``full'' fractional Laplacian on $\R^n$, and in turn by the nonlocal Hardy-type term's dependence on $\Omega$.

As mentioned earlier, it was proved in \cite{FJX}  that if $ n\geq 4\sigma $, $ 1/2<\sigma<1 $, and $\Omega$ is a smooth bounded domain such that $B_\va^+\subset\Omega\subset \R^n_+$ for some $\va>0$, then $
S_{n,\sigma}(\om)<S_{n,\sigma}(\R^n_+),
$
and consequently, $S_{n,\sigma}(\om)$ is achieved.  The assumption that $B_{\va}^+\subset\Omega $ means that the boundary of $ \Omega $ near the origin is flat. In this paper, we would like explore the strict inequality $S_{n,\sigma}(\om)<S_{n,\sigma}(\R^n_+)$ for non-flat domains.

\begin{thm}\label{thm:generaldomains}
Let $ n\geq 4 $, $ \frac{1}{2}<\sigma<1 $ and $ \Omega\subset\mathbb{R}^n $ be an open set. Suppose there exists a point $ a\in \partial\Omega $ such that $\partial\Omega$ is $C^3$ near the point $a$. Then there exist two positive constants $c$ and $C$, both of which depend only on $n,\sigma$ and $\Omega$, such that 
\[
S_{n,\sigma}(\Omega)\le S_{n,\sigma}(\mathbb{R}_+^n)-\frac{c \Gamma_0 H(a) }{\lambda}+C\lambda^{-2\sigma}
\]
for all large $\lambda$, where $H(a)$ is the mean curvature of $\partial \Omega$ at $a$, and \begin{align}\label{eq:integralsign}
\Gamma_0:=\iint_{\mathbb{R}_+^n\times\mathbb{R}_+^n}\frac{(\xi_n-\zeta_n)(|\xi'|^2-|\zeta'|^2)|\Theta(\xi)-\Theta(\zeta)|^2}{|\xi-\zeta|^{n+2\sigma+2}}\ud\xi \ud\zeta<+\infty
\end{align}
with $\Theta$ being a minimizer of $S_{n,\sigma}(\mathbb{R}_+^n)$ that is radially symmetric in the first $n-1$ variables. In particular, $S_{n,\sigma}(\Omega)\le S_{n,\sigma}(\mathbb{R}_+^n)$.
\end{thm}
 We do not know what the sign of $\Gamma_0$ is or whether it is zero, and we leave it as an open question.  We do not have an explicit form of $\Theta$. Some of its estimates are given in Proposition \ref{lem:minimizerbound}.

Since $ S_{n,\sigma}(\Omega) $ is preserved under reflections, rotations, translations and dilations, we can assume that $ a$ is the origin $0$. The smoothness condition assumed in Theorem \ref{thm:generaldomains} indicates that if the principal curvatures of $ \partial\Omega $ at $ 0 $ are denoted as $ \alpha_{i}\ (i=1,2,...,n-1) $, then the boundary $ \partial\Omega $ near  the origin can be represented (up to rotating coordinates if necessary) by
\begin{align}\label{eq:graphfunction}
x_n=h(x')=\frac{1}{2}\sum_{i=1}^{n-1}\alpha_i x_i^2+g(x')|x'|^2,
\end{align}
where $ g $ is a bounded Lipschitz continuous function of the $x'$ variables defined in a small ball in $\R^{n-1}$ such that $g(0)=0$.  To prove Theorem \ref{thm:generaldomains}, we first flatten the boundary near the point $ a $, and then we use a cut-off of a rescaled minimizer of $S_{n,\sigma}(\R^n_+)$ as a test function.

To prove the strict inequality $ S_{n,\sigma}(\Omega)< S_{n,\sigma}(\mathbb{R}_{+}^n) $ without knowing  the sign of $\Gamma_0$, we need a global smallness condition, that is, we need to assume that part of the boundary $\Omega$ is represented by the function in \eqref{eq:graphfunction} with small $\alpha_i$, and $\Omega$ is above its graph.
\begin{thm}\label{thm:generaldomains2}
Let $ n\geq 4 $, $ \frac{1}{2}<\sigma<1 $, $\alpha_1,\cdots,\alpha_{n-1}$ be real numbers, $ g $ be a locally Lipschitz continuous function on $\R^{n-1}$  such that $g(0)=0$,  $h$ be defined as in \eqref{eq:graphfunction}, and $$\mathscr{R}:=\{x=(x',x_n)\in \R^n: x_n>h(x')\}.$$ Let $ \Omega\subset\R^n$ be an open set such that for some $\delta_0>0$ and $R_0>0$, 
\[
(B_{\delta_0}\cap \mathscr{R}) \subset\Omega \subset (\{x\in\R^{n}: |x'|< R_0\}\cap \mathscr{R}).
\]
Then there exists a positive constant $\varepsilon_0$ depending only on $n,\sigma,R_0$ and $\delta_0$ such that if 
\[
|\nabla_{x'}g(x')|\le \varepsilon_0\ \mbox{ for every } |x'|<R_0, \mbox{ and } |\alpha_i|\le\varepsilon_0 \ \mbox{ for every } i=1,\cdots, n-1,
\]
then
 \[
 S_{n,\sigma}(\Omega)< S_{n,\sigma}(\mathbb{R}_{+}^n) .
 \]
\end{thm}

An important intermediate step in the proof of Theorem \ref{thm:generaldomains} is that the minimizers of $ S_{n,\sigma}(\mathbb{R}_+^n) $ are radially symmetric in the first $ n-1 $ variables.
\begin{thm}\label{thm:minimizersymmetry}
Assume that $ n\geq 2 $, $ 1/2<\sigma<1 $ and $ u\in \mathring{H}^{\sigma}(\mathbb{R}_+^n) $ is a minimizer of $S_{n,\sigma}(\R^n_+)$. Then $u$ must be radially symmetric about some point in $ \mathbb{R}^{n-1} $ for the first $ n-1 $ variables.
\end{thm}

In fact, this symmetry holds not only for the minimizers of $S_{n,\sigma}(\R^n_+)$, but also for the solutions of its Euler-Lagrange equation.

\begin{thm}\label{thm:symmetry}
Assume that $ n\geq 2 $, $ 1/2<\sigma<1 $ and $ u\in \mathring{H}^{\sigma}(\mathbb{R}_+^n) $ is a nonnegative solution of 
\begin{align}\label{eq:global}
(-\Delta)_{\mathbb{R}_+^n}^{\sigma}u=u^{\frac{n+2\sigma}{n-2\sigma}} \quad\mbox{in }\mathbb{R}_+^n,
\end{align}
then $u$ must be radially symmetric about some point in $ \mathbb{R}^{n-1} $ for the first $ n-1 $ variables.
\end{thm}
To prove Theorem \ref{thm:symmetry}, we use the method of moving planes for the regional fractional Laplacian. In this step, we adapt ideas in \cite{CLL} for the full fractional Laplacian $ (-\Delta)^{\sigma} $ to the regional fractional Laplacian $ (-\Delta)_{\mathbb{R}_+^n}^{\sigma} $ in our case. Although Theorem \ref{thm:minimizersymmetry} follows from  Theorem \ref{thm:symmetry}, we also provide a proof using the rearrangement arguments, which are of independent interests.

This paper is organized as follows. In Section \ref{sec2}, we prove the radial symmetry in Theorem \ref{thm:minimizersymmetry} and Theorem \ref{thm:symmetry}. In Section \ref{sec3}, we show the properties of the sharp constant $S_{n,\sigma}(\Omega)$ stated in Theorems \ref{thm:generaldomains} and \ref{thm:generaldomains2}. In the Appendix \ref{sec:appendix}, we include the technical calculations for some quantitative integrals of the minimizers $\Theta$ of $S_{n,\sigma}(\R^n_+)$.

\section{Radial symmetry}\label{sec2}

Let $n\ge 2$. If $u$ is a function on $\R^n_+$ and such that for a.e. $x_n\in\R_+$ and every $\lambda>0$ one has $|\{x'\in\R^{n-1}:  |u(x',x_n)|>\lambda \}|<\infty$, where $|\cdot|$ denotes the Lebesgue measure, then we define its rearrangement
$$
u^\sharp(x',x_n):= u(\cdot,x_n)^*(x') \,.
$$
Here $*$ denotes symmetric decreasing rearrangement in $\R^{n-1}$.

\begin{prop}\label{prop:rearrangement}
	Let $n\geq 2$ and $\sigma\in(0,1)$. Then for any $u\in\mathring H^\sigma(\R^n_+)$, one has
	$$
	I_{n,\sigma,\R^n_+}[u] \geq I_{n,\sigma,\R^n_+}[u^\sharp] \,.
	$$
	If the equality holds, then there is an $a'\in\R^{n-1}$ such that either
	$$
	u(x',x_n) = u^\sharp(x'-a',x_n)
	\qquad\text{for a.e.}\ (x',x_n)\in\R^n_+
	$$
	or
	$$
	u(x',x_n) = - u^\sharp(x'-a',x_n)
	\qquad\text{for a.e.}\ (x',x_n)\in\R^n_+ \,.
	$$
\end{prop}

\begin{proof}[Proof of Theorem \ref{thm:minimizersymmetry}] 
By the equimeasurability property of symmetric decreasing rearrangment in $\R^{n-1}$ we have $\int_{\R^{n-1}} (u^\sharp(x',x_n))^p \,dx' = \int_{\R^{n-1}} |u(x',x_n)|^p \,dx'$ for any $p>0$. Thus, as a consequence of Proposition \ref{prop:rearrangement}, we infer that any minimizer $u$ of $S_{n,\sigma}(\R^n_+)$ satisfies either $u(x',x_n) = u^*(x'-a',x_n)$ for a.e. $(x',x_n)\in\R^n_+$ or $u(x',x_n) = -u^*(x'-a',x_n)$ for a.e. $(x',x_n)\in\R^n_+$, for some $a'\in\R^{n-1}$.
\end{proof}

For the proof of the inequality in Proposition \ref{prop:rearrangement}, we use an argument due to Almgren-Lieb \cite{AL}. To characterize the cases of equality, we use a strengthening of this argument due to Frank-Seiringer \cite{FS}.

\begin{proof}[Proof of Proposition \ref{prop:rearrangement}] 
	We write
	$$
	I_{n,\sigma,\R^n_+}[u] = \iint_{\R_+\times\R_+}  J_{|x_n-y_n|}[u(\cdot,x_n),u(\cdot,y_n)] \,\ud x_n\,\ud y_n
	$$
	with
	$$
	J_{r}[f,g] := \iint_{\R^{n-1}\times\R^{n-1}} \frac{(f(x')-g(y'))^2}{(|x'-y'|^2 +r^2)^\frac{n+2\sigma}{2}}\,\ud x'\,\ud y' \,.
	$$
	Note that when $r>0$, the kernel $(|z'|^2+r^2)^{-\frac{n+2\sigma}{2}}$ is integrable. Therefore, we can expand the square $(f(x')-f(y'))^2$ and, in the ``diagonal terms" perform one of the integrals, which leads to the square of the $L^2$-norms of $f$ and $g$. Since these norms coincide with those of $f^*$ and $g^*$, we obtain
	$$
	J_{r}[f,g] - J_{r}[f^*,g^*] = 2 \iint_{\R^{n-1}\times\R^{n-1}} \frac{f^*(x')\,g^*(y') - f(x')g(y')}{(|x'-y'|^2 +r^2)^\frac{n+2\sigma}{2}}\,\ud x'\,\ud y' \,.
	$$
	By the Riesz rearrangement inequality (see, e.g., Theorem 3.7 in Lieb--Loss \cite{LiebLoss}),
	$$
	J_{r}[f,g] - J_{r}[f^*,g^*] \geq 0 \,.
	$$
	Inserting this with $f=u(\cdot,x_n)$ and $g=u(\cdot,y_n)$ into the above formula we obtain $I_{n,\sigma,\R^n_+}[u] \geq I_{n,\sigma,\R^n_+}[u^\#]$, as claimed.
	
	In the above argument, we use the square integrability of $u(\cdot,x_n)$ for a.e. $x_n$, which is not a priori clear. We can argue more carefully as follows. We first observe that $(u(x)-u(y))^2 \geq (|u(x)|-u(y)|)^2$, so $I_{n,\sigma,\R^n_+}[u]\geq I_{n,\sigma,\R^n_+}[|u|]$. Now we apply the above argument to $\min\{ (|u|-\epsilon)_+,M\}$ with two positive constants $\epsilon$ and $M$, which belongs to $L^\infty$ and has support on a set of finite measure, so is in $L^2$. So for this cut off function we have the claimed inequality and then we can remove the cut-offs by applying the monotone convergence theorem.
	
	Now assume that we have the equality $I_{n,\sigma,\R^n_+}[u] = I_{n,\sigma,\R^n_+}[u^\#]$. Then we also must have the equality $I_{n,\sigma,\R^n_+}[u]= I_{n,\sigma,\R^n_+}[|u|]$ and, by the above argument we easily see that either $u(x)=|u(x)|$ for a.e. $x\in\R^n_+$ or $u(x)=-|u(x)|$ for a.e. $x\in\R^n_+$. Next, the equality $I_{n,\sigma,\R^n_+}[|u|] = I_{n,\sigma,\R^n_+}[u^\#]$ implies that for a.e.\ $(x_n,y_n)\in\R_+\times\R_+$,
	$$
	J_{|x_n-y_n|}[|u(\cdot,x_n)|,|u(\cdot,y_n)|] = J_{|x_n-y_n|}[u^\sharp(\cdot,x_n),u^\sharp(\cdot,y_n)] \,.
	$$
	Thus, by Lieb's theorem (see, e.g., Theorem 3.9 in Lieb--Loss \cite{LiebLoss}) for a.e.\ $(x_n,y_n)\in\R_+\times\R_+$ there is an $a'(x_n,y_n)\in\R^{n-1}$ such that
	$$
	|u(x',x_n)| = u^\sharp(x'-a'(x_n,y_n),x_n)
	\qquad\text{and}\qquad
	|u(y',y_n)| = u^\sharp(y'-a'(x_n,y_n),y_n)
	$$
	for a.e.\ $x',y'\in\R^{n-1}$. Since the left hand side in the first equation is independent of $y_n$ and in the second one of $x_n$, we deduce that $a'(x_n,y_n)$ is independent of $x_n$ and $y_n$, that is, it is a constant $a'\in\R^{n-1}$. This implies the assertion of the proposition.
\end{proof}

Next, we will prove Theorem \ref{thm:symmetry} using the method of moving planes. 

\begin{prop}\label{lem:minimizerbound}
Let $ n\geq 2 $ and $ 1/2<\sigma<1 $. Let $ 0\not\equiv u\in \mathring{H}^{\sigma}(\mathbb{R}_+^n) $ be non-negative and satisfy \eqref{eq:global}. Then $u\in C^{2\sigma-1}_{loc}(\overline\R^n_+)\cap C^\infty(\R^n_+)$, and there are constants $ 0<c\leq C<+\infty $ (depending on $ u $) such that
\begin{align}\label{eq:minimizerbound}
c\frac{x_n^{2\sigma-1}}{(1+|x|)^{n+2\sigma-2}}\leq u(x)\leq C\frac{x_n^{2\sigma-1}}{(1+|x|)^{n+2\sigma-2}}.
\end{align}
Furthermore, $x_n^{1-2\sigma}u(x) \in C^1(\overline\R^n_+)$ and there is a constant $\widetilde C>0$ (depending on $u$) such that
\begin{align}\label{eq:gradientbound}
|\nabla (x_n^{1-2\sigma}u(x))|\le \frac{\widetilde C}{(1+|x|)^{n+2\sigma-1}}.
\end{align}
for $ x\in\mathbb{R}_+^n $.
\end{prop}
\begin{proof}
The estimate \eqref{eq:minimizerbound} was proved in Proposition 1.5 in \cite{FJX}. The estimate \eqref{eq:gradientbound} for $|x|\le 1$ follows from \eqref{eq:minimizerbound} and the regularity that $x_n^{1-2\sigma}u(x)\in C^1(\overline \R^n_+)$ proved in Fall-Ros-Oton \cite{FRO}. Let $$\tilde u(x)=|x|^{2\sigma-n}u\left(\frac{x}{|x|^2}\right).$$ Then $\tilde u$ satisfies \eqref{eq:global} as well. Thus, $\tilde u$ satisfies \eqref{eq:minimizerbound} and $|\nabla (x_n^{1-2\sigma}\tilde u(x))|\le C$ in $\overline B_1^+$ for some $C>0$ (depending on $u$). 
Since $$u(x)=|x|^{2\sigma-n}\tilde u\left(\frac{x}{|x|^2}\right)$$ as well, we have 
$$x_n^{1-2\sigma}u(x)=|x|^{2-2\sigma-n}y_n^{1-2\sigma}\tilde u\left(y\right),\quad\mbox{where}\quad y=\frac{x}{|x|^2}.$$
The estimate \eqref{eq:gradientbound} for $|x|\ge 1$ follows from that $|y_n^{1-2\sigma}\tilde u(y)|+|\nabla (y_n^{1-2\sigma}\tilde u(y))|\le C$ in $\overline B_1^+$.
\end{proof}

\begin{proof}[Proof of Theorem \ref{thm:symmetry}]
For $ \lambda\in\mathbb{R} $ we define
\begin{align}
T_{\lambda}=\left\{x\in\mathbb{R}_+^n:x_1=\lambda\right\},\quad& x^{\lambda}=(2\lambda-x_1,x_2,\cdots,x_n)\nonumber\\
u_{\lambda}(x)=u(x^{\lambda}),\quad& w_{\lambda}(x)=u_{\lambda}(x)-u(x)\nonumber\nonumber
\end{align}
and
\begin{align}
\Sigma_{\lambda}=\left\{x\in\mathbb{R}_+^n:x_1<\lambda\right\},\quad\widetilde{\Sigma}_{\lambda}=\left\{x^{\lambda}:x\in\Sigma_{\lambda}\right\}\nonumber\nonumber.
\end{align}
By Proposition \ref{lem:minimizerbound}, we have $ \lim\limits_{|x|\rightarrow\infty,\ x\in\mathbb{R}_+^n}\omega_{\lambda}(x)=0 $ for any fixed $ \lambda $.  Hence, if $ \omega_{\lambda} $ is negative somewhere in $ \Sigma_{\lambda} $, then the minimum of $ \omega_{\lambda} $ in $\overline\Sigma_{\lambda} $ would be attained in $ \Sigma_{\lambda} $. Let
\begin{align}
\Sigma_{\lambda}^-=\left\{x\in\Sigma_{\lambda}:\omega_{\lambda}(x)<0\right\}\nonumber\nonumber.
\end{align}
Then for $x\in \Sigma_{\lambda}^{-} $,  we have
\begin{align*}
(-\Delta)_{\R^n_+}^{\sigma}w_{\lambda}(x)&=u_{\lambda}(x)^{\frac{n+2\sigma}{n-2\sigma}}-u(x)^{\frac{n+2\sigma}{n-2\sigma}}\nonumber\\
&=\frac{n+2\sigma}{n-2\sigma}\left(\int_0^1 \left(tu_{\lambda}(x)+(1-t)u(x)\right)^{\frac{4\sigma}{n-2\sigma}}\,\ud t\right) \, w_\lambda(x)\\
&\geq\frac{n+2\sigma}{n-2\sigma}u(x)^{\frac{4\sigma}{n-2\sigma}} w_\lambda(x).
\end{align*}
That is
\begin{equation}\label{eq:movingsymmetry}
(-\Delta)^{\sigma}_{\R^n_+}w_{\lambda}(x)+c(x)\omega_{\lambda}(x)\geq 0\quad\mbox{in } \Sigma_{\lambda}^{-} ,
\end{equation}
where $$ c(x):=-\frac{n+2\sigma}{n-2\sigma}u(x)^{\frac{4\sigma}{n-2\sigma}}.$$ Also, by Proposition \ref{lem:minimizerbound}, we have 
\begin{align}
|x|^{2\sigma}|c(x)|\le\frac{C}{|x|^{\frac{4\sigma^2+2n\sigma-4\sigma}{n-2\sigma}}}\nonumber,
\end{align} 
and thus,   $$ \liminf_{|x|\rightarrow\infty,\ x\in\mathbb{R}_+^n}|x|^{2\sigma}c(x)=0. $$ 

Let $x^0\in \Sigma_{\lambda}$ be such that $w(x^0)=\min_{\overline\Sigma_\lambda} w<0$. Then
\begingroup
\allowdisplaybreaks
\begin{align}
(-\Delta)^{\sigma}_{\R^n_+}w_\lambda(x^0)&=2P.V.\int_{\mathbb{R}_+^n}\frac{w_\lambda(x^0)-w_\lambda(y)}{|x^0-y|^{n+2\sigma}}dy\nonumber\\
&=2P.V.\left\{\int_{\Sigma_\lambda}\frac{w_\lambda(x^0)-w_\lambda(y)}{|x^0-y|^{n+2\sigma}}dy+\int_{\widetilde{\Sigma_\lambda}}\frac{w_\lambda(x^0)-w_\lambda(y)}{|x^0-y|^{n+2\sigma}}dy\right\}\nonumber\\
&=2P.V.\left\{\int_{\Sigma_\lambda}\frac{w_\lambda(x^0)-w_\lambda(y)}{|x^0-y|^{n+2\sigma}}dy+\int_{\Sigma_\lambda}\frac{w_\lambda(x^0)-w_\lambda(y^\lambda)}{|x^0-y^\lambda|^{n+2\sigma}}dy\right\}\nonumber\\
&\le 2P.V.\left\{\int_{\Sigma_\lambda}\frac{w_\lambda(x^0)-w_\lambda(y)}{|x^0-y^\lambda|^{n+2\sigma}}dy+\int_{\Sigma_\lambda}\frac{w_\lambda(x^0)+w_\lambda(y)}{|x^0-y^\lambda|^{n+2\sigma}}dy\right\}\nonumber\\
&=4\int_{\Sigma_\lambda}\frac{w_\lambda(x^0)}{|x^0-y^\lambda|^{n+2\sigma}}dy.\label{eq:fractionalcalculation}
\end{align}
\endgroup
Moreover, if $|x_0|>|\lambda|$ is sufficiently large, then 
\begin{align}
\int_{\Sigma_\lambda}\frac{1}{|x^0-y^\lambda|^{n+2\sigma}}dy&\geq\int_{\{y\in \widetilde\Sigma_\lambda: 2|x_0|\le |y-x^0|\le 3|x_0|}\frac{1}{|x^0-y|^{n+2\sigma}}dy\nonumber\\
&\geq\frac{m}{|x^0|^{2\sigma}},\nonumber
\end{align}
where $m>0$ is a constant.  Together with \eqref{eq:movingsymmetry}, we obtain
\begin{align}\label{eq:contradiction1}
0\leq(-\Delta)_{\R^n_+}^{\sigma}u(x^0)+c(x^0)u(x^0)\leq\left[\frac{m}{|x^0|^{2\sigma}}+c(x^0)\right]u(x^0)<0,
\end{align}
which is a contradiction. 

This proves that if $\lambda$ is sufficiently negative, then 
\[
w_\lambda\ge 0\quad\mbox{in }\Sigma_\lambda.
\]
Therefore, we can define
\begin{align}
\bar\lambda=\sup\left\{\lambda\in\mathbb{R}:w_{\mu}(x)\geq 0,\,\,\,\,\,\,\forall x\in \Sigma_{\lambda},\mu\leq\lambda\right\}.\nonumber
\end{align}

If $ \bar\lambda=+\infty $, then since $u(x)\to 0$ as $|x|\to\infty$, we have that $u\equiv 0$, which is a contradiction.  Hence, $\bar\lambda<\infty$. We will prove in the below that $w_{\bar\lambda}\equiv 0$ in $\Sigma_{\bar\lambda}$.

We argue by contradiction that we suppose $w_{\bar\lambda}> 0$ at some point, and thus in some open subset of $\Sigma_{\bar\lambda}$. 

Then $w_{\bar\lambda}> 0$ in  $\Sigma_{\bar\lambda}$, since otherwise, if there exists $z\in \Sigma_{\bar\lambda}$ such that $w_{\bar\lambda}(z)=0$, then by the equation of $w_{\bar\lambda}$, it follows that
\begin{align}
0=(-\Delta)^{\sigma}_{\R^n_+}w_{\bar\lambda}(z)&=2P.V.\int_{\mathbb{R}_+^n}\frac{w_{\bar\lambda}(z)-w_{\bar\lambda}(y)}{|z-y|^{n+2\sigma}}dy\nonumber\\
&=2P.V.\left\{\int_{\Sigma_{\bar\lambda}}\frac{0-w_{\bar\lambda}(y)}{|z-y|^{n+2\sigma}}dy+\int_{\widetilde{\Sigma_{\bar\lambda}}}\frac{0-w_{\bar\lambda}(y)}{|z-y|^{n+2\sigma}}dy\right\}\nonumber\\
&=2P.V.\int_{\Sigma_{\bar\lambda}}w_\lambda(y)\left(\frac{1}{|z-y^{\bar\lambda}|^{n+2\sigma}}-\frac{1}{|z-y|^{n+2\sigma}}\right)dy\nonumber\\
&<0,\nonumber
\end{align}
which is a contradiction.

Now, from \eqref{eq:contradiction1}, we have that there exists $R_0>0$ such that for every $\lambda\in[\bar\lambda,\bar\lambda+1]$,
\[
w_\lambda\ge 0\quad\mbox{in }\Sigma_\lambda\setminus B_{R_0}.
\]
Since we just proved that $w_{\bar\lambda}> 0$ in  $\Sigma_{\bar\lambda}$, by continuity, we have that for every $\va>0$, there exists $\delta>0$ such that
\[
w_\lambda>0\quad\mbox{in } \overline{B_{R_0}\cap \Sigma_{\bar\lambda-\va}\cap\{x_n> \va\} }\mbox{ for all }\lambda\in[\bar\lambda,\bar\lambda+\delta].
\]

We are going to show that 
\begin{equation}\label{eq:movingfurther}
w_\lambda\ge 0\quad\mbox{in } \Sigma_{\lambda}\mbox{ for all }\lambda\in[\bar\lambda,\bar\lambda+\delta]
\end{equation}
if we choose $\va$ and $\delta$ to be small enough. Suppose there exists $\bar x$ satisfying $\bar x_n\in(0,\va )$ or $\bar x_1\in(\bar\lambda-\va,\lambda)$ such that
\[
w_\lambda(\bar x)=\min_{\overline\Sigma_{\lambda}} w_\lambda<0.
\]
Then from \eqref{eq:movingsymmetry} and \eqref{eq:fractionalcalculation}, we have
\[
\frac{n+2\sigma}{n-2\sigma}u(\bar x)^{\frac{4\sigma}{n-2\sigma}} w_\lambda(\bar x)\le (-\Delta)^{\sigma}_{\R^n_+}w_\lambda(\bar x)\le 4\int_{\Sigma_\lambda}\frac{2w_\lambda(\bar x)}{|\bar x-y^\lambda|^{n+2\sigma}}dy.
\]
That is,
\begin{equation}\label{eq:contradiction2}
\int_{\Sigma_\lambda}\frac{1}{|\bar x-y^\lambda|^{n+2\sigma}}dy\le Cu(\bar x)^{\frac{4\sigma}{n-2\sigma}}\le  C \bar x_n^{\frac{4\sigma(2\sigma-1)}{n-2\sigma}}.
\end{equation}

If $\bar x_1\in(\bar\lambda-\va,\lambda)$, then 
\[
\int_{\Sigma_\lambda}\frac{1}{|\bar x-y^\lambda|^{n+2\sigma}}dy\ge \frac{C}{(\va+\delta)^{2\sigma}} \to \infty \quad\mbox{as }\va+\delta\to 0, 
\]
contradicting to \eqref{eq:contradiction2} since $\bar x\in B_{R_0}$.

If $\bar x_n\in(0,\va )$, then since $\bar x\in B_{R_0}$, we have 
\[
\int_{\Sigma_\lambda}\frac{1}{|\bar x-y^\lambda|^{n+2\sigma}}dy\ge C, 
\]
contradicting to \eqref{eq:contradiction2} if $\va$ is small.

This proves \eqref{eq:movingfurther}, which contradicts with the definition of $\bar\lambda$. Hence,  we have proved that $w_{\bar\lambda}\equiv 0$ in $\Sigma_{\bar\lambda}$, that is, $ u $ is symmetric about the plane $ T_{\bar\lambda} $ in $ \mathbb{R}_+^n $. Since the $ x_1 $ direction can be chosen arbitrarily for the first $ n-1 $ variables, we have actually shown that $ u $ is radially symmetric with respect to some point in $\pa\R^n_+$ in the first $ n-1 $ variables.
\end{proof}

\begin{prop}\label{prop:integrability}
Assume that $ n\geq 4$, $ 1/2<\sigma<1 $, $\lambda>0$,  and $\gamma>0$ that $\gamma\neq2\sigma$. Let $\Theta$ be a minimizer of $S_{n,\sigma}(\R^n_+)$ that is radially symmetric in the first $n-1$ variables. Then
\begin{align}
\iint_{B_\lambda^+\times B_\lambda^+}\frac{ (|\xi'|^2+|\zeta'|^2)^{\gamma/2} |\Theta(\xi)-\Theta(\zeta)|^2}{|\xi-\zeta|^{n+2\sigma}}d\xi d\zeta&<C(n,\sigma,\gamma)(1+\lambda^{\gamma-2\sigma}).\label{eq:integrability2}
\end{align}
If $\gamma<2\sigma$, then 
\begin{align}
 \iint_{(\R^n_+\times\R^n_+)\setminus(B_\lambda^+\times B_\lambda^+)}\frac{(|\xi'|^2+|\zeta'|^2)^{\gamma/2}|\Theta(\xi)-\Theta(\zeta)|^2}{|\xi-\zeta|^{n+2\sigma}}\,\ud\xi \ud\zeta &\le C(n,\sigma)\lambda^{\gamma-2\sigma}.\label{eq:integralsignest}
\end{align}

\end{prop}
The proof of this proposition is given in the Appendix \ref{sec:appendix}.

\section{Sharp constants}\label{sec3}

\begin{proof} [Proof of Theorem \ref{thm:generaldomains2}]
First of all, we observe that $ S_{n,\sigma}(\Omega) $ is preserved under reflections, rotations, translations and dilations. Hence, we can assume $\delta_0=4$.

Let $ \Phi:\Omega \rightarrow\mathbb{R}^n $ defined as $$ \xi=\Phi(x)=(x_1,\cdots, x_{n-1}, x_n-h(x')). $$ 
Let $\Theta$ be a minimizer of $S_{n,\sigma}(\R^n_+)$ that is radially symmetric in the first $n-1$ variables. For every $\lambda>0$, we let
\begin{align}
\Theta_{\lambda}(x)=\lambda^{\frac{n-2\sigma}{2}}\Theta({\lambda}x)\nonumber\nonumber.
\end{align}
Let $ \eta $ be a cut off function such that $ \eta\in C^{1}(\mathbb{R}^n) $, $ \eta\equiv 1 $ in $ B_{2} $, $ 0\leq \eta\leq 1 $ in $ B_{3} $ and $ \eta\equiv 0 $ in $ B_{3}^{c} $.  Let $ \theta_{\lambda}(x)=(\eta\Theta_{\lambda})(x) $ and $ v_{\lambda}=\theta_{\lambda}\circ\Phi(x) $. Then we have      
\begin{align}
I_{n,\sigma,\Omega}[v_{\lambda}]&=\iint_{\Omega \times\Omega }\frac{|v_{\lambda}(x)-v_{\lambda}(y)|^2}{|x-y|^{n+2\sigma}}\,\ud x \ud y\nonumber\\
&=\iint_{U\times U}\frac{|\theta_{\lambda}(\xi)-\theta_{\lambda}(\zeta)|^2}{\left(|\xi'-\zeta'|^2+\left[\xi_n+h(\xi')-\zeta_n-h(\zeta')\right]^2\right)^{\frac{n+2\sigma}{2}}}d\xi d\zeta,\label{eq:mainchangeofvariable}
\end{align}
where $U=\Phi( \Omega) $. We would like to  analyze the denominator 
\begin{align}
A(\xi,\zeta)&=\left(|\xi'-\zeta'|^2+\left[\xi_n+h_(\xi')-\zeta_n-h(\zeta')\right]^2\right)^{-\frac{n+2\sigma}{2}}\nonumber\\
&=|\xi-\zeta|^{-(n+2\sigma)}[1+B(\xi,\zeta)+C(\xi,\zeta)+D(\xi,\zeta)]^{-\frac{n+2\sigma}{2}}\label{eq:denominator},
\end{align}
where
\begingroup
\allowdisplaybreaks
\begin{align}
B(\xi,\zeta)&=\frac{1}{|\xi-\zeta|^2}(\xi_n-\zeta_n)\left(\sum_{i=1}^{n-1}\alpha_i \xi_i^2-\sum_{i=1}^{n-1}\alpha_i \zeta_i^2\right),\nonumber\\
C(\xi,\zeta)&=\frac{2(\xi_n-\zeta_n)}{|\xi-\zeta|^2}\left(g\left(\xi'\right)|\xi'|^2-g\left(\zeta'\right)|\zeta'|^2\right),\nonumber\\
D(\xi,\zeta)&=\frac{(h(\xi')-h(\zeta'))^2}{|\xi-\zeta|^2}\nonumber\nonumber.
\end{align}
\endgroup
We will show that each term in the above is sufficiently small so that we can have a Taylor expansion for $A(\xi,\zeta)$.

For $B(\xi,\zeta)$, since
\begin{align}
\left|\sum_{i=1}^{n-1}\alpha_i(\xi_i^2-\zeta_i^2)\right|&\le \left(\sum_{i=1}^{n-1}\alpha_i^2(\xi_i+\zeta_i)^2\right)^{1/2}\left(\sum_{i=1}^{n-1}(\xi_i-\zeta_i)^2\right)^{1/2}\nonumber\\
&\le 2\varepsilon_0 (|\xi'|^2+|\zeta'|^2) ^{1/2}  |\xi'-\zeta'|, \nonumber
\end{align}
where we used $|\alpha_i|\le \varepsilon_0$ for every $i=1,\cdots,n-1$, then we have
\begingroup
\allowdisplaybreaks
\begin{align}
|B(\xi,\zeta)|&\le \varepsilon_0 (|\xi'|^2+|\zeta'|^2) ^{1/2} .\label{eq:estB}
\end{align}
For $C(\xi,\zeta)$ and $ D(\xi,\zeta) $, since
\begin{align}
&\left|g\left(\xi'\right)|\xi'|^2-g\left(\zeta'\right)|\zeta'|^2\right|\nonumber\\
& \le \left|g\left(\xi'\right)|\xi'|^2-g\left(\zeta'\right)|\xi'|^2\right|+\left|g\left(\zeta'\right)|\xi'|^2-g\left(\zeta'\right)|\zeta'|^2\right|\nonumber\\
&\le \varepsilon_0 |\xi'-\zeta'| |\xi'|^2+ \varepsilon_0 |\zeta'|\left(|\xi'|+|\zeta'|\right)\cdot \left(|\xi'|-|\zeta'|\right)\nonumber\\
&\le \left[ \varepsilon_0|\xi'|^2+  \varepsilon_0|\zeta'|(|\xi'|+|\zeta'|)\right]|\xi'-\zeta'|\nonumber\\
&\le  \frac{3\varepsilon_0 (|\xi'|^2 + |\zeta'|^2)}{2}\cdot |\xi'-\zeta'|, \nonumber
\end{align}
\endgroup
where we used $g(0)=0$ and $|\nabla_{x'}g(x')|\le \varepsilon_0$, then we have
\begin{align}
|C(\xi,\zeta)|&\leq \frac{3 \varepsilon_0 (|\xi'|^2 + |\zeta'|^2)}{2}\label{eq:estC}
\end{align}
and
\begingroup
\allowdisplaybreaks
\begin{align}
|D(\xi,\zeta)|&=\frac{\left(h(\xi')-h(\zeta')\right)^2}{|\xi-\zeta|^2}\nonumber\\
&\le\frac{2\left(\frac{1}{2}\sum_{i=1}^{n-1}\left(\alpha_i\xi_i^2-\alpha_i\zeta_i^2\right)\right)^2}{|\xi-\zeta|^2}+\frac{2\left(g\left( \xi' \right)|\xi'|^2- g\left( \zeta' \right)|\zeta'|^2\right)^2}{|\xi-\zeta|^2}\nonumber\\
&\leq\frac{(|\xi'|^2+|\zeta'|^2)\left(\sum_{i=1}^{n-1}\alpha_i^2\right) |\xi'-\zeta'|^2}{|\xi-\zeta|^2}+\frac{2\left(\frac{3\varepsilon_0 (|\xi'|^2 + |\zeta'|^2)}{2}|\xi'-\zeta'|\right)^2}{ |\xi-\zeta|^2}\nonumber\\
&\leq (n-1)\varepsilon_0^2 (|\xi'|^2+|\zeta'|^2)+\frac{9 \varepsilon_0^2}{2} (|\xi'|^2+|\zeta'|^2)^2.\label{eq:estD}
\end{align}
\endgroup
Hence, for every $(\xi,\zeta)\in U\times U$, which satisfies $|\xi'|<R_0$ and $|\zeta'|<R_0$, there holds
\[
B(\xi,\zeta) + C(\xi,\zeta) +D(\xi,\zeta) \le 2R_0\varepsilon_0+ 3R_0^2\varepsilon_0 + 2(n-1)R_0^2\varepsilon_0^2 + 18R_0^4\varepsilon_0^2.
\] 
Thus, if we choose $ \varepsilon_0 $ to be sufficiently small, each of $B(\xi,\zeta), C(\xi,\zeta)$ and $D(\xi,\zeta)$ is small, so that we can have the Taylor expansion of $ A(\xi,\zeta) $. To be more explicitly, first we can choose a proper $ \va_0 $ such that $ |B(\xi,\zeta)|+|C(\xi,\zeta)|+|D(\xi,\zeta)|<\frac{1}{2} $ for all $ (\xi,\zeta)\in U \times U $. Then we can choose a constant $ A_1>0 $ such that when $ |a|<1/2 $,
\begin{align}
(1+a)^{-\frac{n+2\sigma}{2}}\leq 1-\frac{n+2\sigma}{2}a+A_1 a^2.\nonumber
\end{align}
Therefore, if we denote $ E(\xi,\zeta) := B(\xi,\zeta)+C(\xi,\zeta)+D(\xi,\zeta) $, then we have
\begin{align}
&A(\xi,\zeta)|\xi-\zeta|^{n+2\sigma}\nonumber\\
&\leq1-\frac{(n+2\sigma)}{2} B(\xi,\zeta)-\frac{n+2\sigma}{2}C(\xi,\zeta) -\frac{n+2\sigma}{2}D(\xi,\zeta)+A_1 E(\xi,\zeta)^2\nonumber\\
&\le 1-\frac{(n+2\sigma)}{|\xi-\zeta|^2}(\xi_n-\zeta_n)\left(\frac{1}{2}\sum_{i=1}^{n-1}\alpha_i\xi_i^2-\frac{1}{2}\sum_{i=1}^{n-1}\alpha_i\zeta_i^2\right)+F(\xi,\zeta),\label{eq:taylorexpansion}
\end{align} 
where 
\begin{align*}
F(\xi,\zeta)&=\frac{n+2\sigma}{2}|C(\xi,\zeta)| +\frac{n+2\sigma}{2}D(\xi,\zeta)+A_1 E(\xi,\zeta)^2.
\end{align*} 

Therefore, it follows from \eqref{eq:mainchangeofvariable}, \eqref{eq:denominator} and \eqref{eq:taylorexpansion} that
\begingroup
\allowdisplaybreaks
\begin{align}
&I_{n,\sigma,\Omega_{\mu}}[v_{\lambda}]\nonumber\\
&=\iint_{U\times U}\frac{|\theta_{\lambda}(\xi)-\theta_{\lambda}(\zeta)|^2}{\left(|\xi'-\zeta'|^2+\left(\xi_n+h_{\mu}(\xi')-\zeta_n-h_{\mu}(\zeta')\right)^2\right)^{\frac{n+2\sigma}{2}}}\,\ud\xi \ud\zeta\nonumber\\
&\leq\iint_{U\times U}\frac{|\theta_{\lambda}(\xi)-\theta_{\lambda}(\zeta)|^2}{|\xi-\zeta|^{n+2\sigma}}\,\ud\xi \ud\zeta\nonumber\\
&\quad- \frac{(n+2\sigma)}{2} \iint_{U\times U}\frac{(\xi_n-\zeta_n)\left(\sum_{i=1}^{n-1}\alpha_i\xi_i^2-\sum_{i=1}^{n-1}\alpha_i\zeta_i^2\right)|\theta_{\lambda}(\xi)-\theta_{\lambda}(\zeta)|^2}{|\xi-\zeta|^{n+2\sigma+2}}\,\ud\xi \ud\zeta\nonumber\\
&\quad+\iint_{U\times U}\frac{F(\xi,\zeta)|\theta_{\lambda}(\xi)-\theta_{\lambda}(\zeta)|^2}{|\xi-\zeta|^{n+2\sigma}}\,\ud\xi \ud\zeta.\label{eq:expansionofI}
\end{align}
\endgroup

We are going to estimate each term in the right hand side of \eqref{eq:expansionofI}.

We start with estimating the third term there. By using \eqref{eq:estB}, \eqref{eq:estC} and \eqref{eq:estD}, there exists a positive constant $C$ which depends only on $n$, $\sigma$ and $R_0$ such that
\begin{align*}
F(\xi,\zeta)& \le C\varepsilon_0 (|\xi'|^2 + |\zeta'|^2)
\end{align*} 
for all $(\xi,\zeta)\in U\times U \subset B_{R_0}^+\times B_{R_0}^+.$ Therefore,
\begingroup
\allowdisplaybreaks
\begin{align*}
&\iint_{U\times U}\frac{F(\xi,\zeta)|\theta_{\lambda}(\xi)-\theta_{\lambda}(\zeta)|^2}{|\xi-\zeta|^{n+2\sigma}}\,\ud\xi \ud\zeta \\
&\le \frac{C\varepsilon_0}{\lambda^{2}} \iint_{B_{\lambda R_0}^+\times B_{\lambda R_0}^+}\frac{(|\xi'|^2 + |\zeta'|^2)|\eta(\lambda^{-1}\xi)\Theta(\xi)-\eta(\lambda^{-1}\zeta)\Theta(\zeta)|^2}{|\xi-\zeta|^{n+2\sigma}}\,\ud\xi \ud\zeta\\
&\le \frac{2C\varepsilon_0}{\lambda^{2}} \iint_{B_{\lambda R_0}^+\times B_{\lambda R_0}^+}\frac{|\zeta'|^2|\eta(\lambda^{-1}\xi)\Theta(\xi)-\eta(\lambda^{-1}\zeta)\Theta(\zeta)|^2}{|\xi-\zeta|^{n+2\sigma}}\,\ud\xi \ud\zeta.
\end{align*}
\endgroup
By the Cauchy-Schwarz inequality,  we have
\begin{align}
&|\eta(\lambda^{-1}\xi)\Theta(\xi)-\eta(\lambda^{-1}\zeta)\Theta(\zeta)|^2\nonumber\\
&\le  2 |\eta(\lambda^{-1}\xi)|^2 |\Theta(\xi)-\Theta(\zeta)|^2 + 2 |\eta(\lambda^{-1}\xi)-\eta(\lambda^{-1}\zeta)|^2 |\Theta(\zeta)|^2\nonumber\\
&\le  2  |\Theta(\xi)-\Theta(\zeta)|^2 + 2 |\eta(\lambda^{-1}\xi)-\eta(\lambda^{-1}\zeta)|^2 |\Theta(\zeta)|^2.\label{eq:CWcutoff}
\end{align}
Since for every $\zeta\in\R^n_+$,
\begingroup
\allowdisplaybreaks
\begin{align*}
&\int_{\R^n_+} \frac{ |\eta(\lambda^{-1}\xi)-\eta(\lambda^{-1}\zeta)|^2}{|\xi-\zeta|^{n+2\sigma}}\,\ud\xi\\
&\le \int_{\{|\xi-\zeta|<\lambda\}} \frac{C}{\lambda^{2}|\xi-\zeta|^{n+2\sigma-2}}\,\ud\xi+\int_{\{|\xi-\zeta|\ge \lambda\}} \frac{4}{|\xi-\zeta|^{n+2\sigma}}\,\ud\xi\\
&\le \frac{C}{\lambda^{2\sigma}},
\end{align*}
\endgroup
we obtain
\begingroup
\allowdisplaybreaks
\begin{align*}
&\iint_{B_{\lambda R_0}^+\times B_{\lambda R_0}^+}  \frac{ |\zeta'|^2 |\eta(\lambda^{-1}\xi)-\eta(\lambda^{-1}\zeta)|^2 |\Theta(\zeta)|^2 }{|\xi-\zeta|^{n+2\sigma-2}}\,\ud\xi\ud\zeta\\
&\le C \int_{B_{\lambda R_0}^+} |\zeta'|^2 |\Theta(\zeta)|^2 \int_{\R^n_+} \frac{ |\eta(\lambda^{-1}\xi)-\eta(\lambda^{-1}\zeta)|^2}{|\xi-\zeta|^{n+2\sigma}}\,\ud\xi\ \ud\zeta\\
&\le C \int_{B_{\lambda R_0}^+}|\zeta'|^2 |\Theta(\zeta)|^2 \,\ud \zeta\\
&\le C(1+\lambda^{4-n})\lambda^{-2\sigma}\\
&\le C,
\end{align*}
\endgroup
where we used Proposition \ref{lem:minimizerbound} and $n\ge 3$. Then by using  \eqref{eq:integrability2}, we obtain that
\begin{align}\label{eq:limitestimate}
\iint_{U\times U}\frac{F(\xi,\zeta)|\theta_{\lambda}(\xi)-\theta_{\lambda}(\zeta)|^2}{|\xi-\zeta|^{n+2\sigma}}d\xi d\zeta\leq \frac{C\varepsilon_0}{\lambda^{2\sigma}}.
\end{align}

Next, we  estimate the second term in the right hand side of \eqref{eq:expansionofI}. For every $i=1,\cdots,n-1$, we have
\begingroup
\allowdisplaybreaks
\begin{align}
&\iint_{U\times U}\frac{(\xi_n-\zeta_n)\left(\sum_{i=1}^{n-1}\alpha_i\xi_i^2-\sum_{i=1}^{n-1}\alpha_i\zeta_i^2\right)|\theta_{\lambda}(\xi)-\theta_{\lambda}(\zeta)|^2}{|\xi-\zeta|^{n+2\sigma+2}}\,\ud\xi \ud\zeta\nonumber\\
& = \iint_{B_1^+\times B_1^+}\frac{(\xi_n-\zeta_n)\left(\sum_{i=1}^{n-1}\alpha_i\xi_i^2-\sum_{i=1}^{n-1}\alpha_i\zeta_i^2\right)|\Theta_{\lambda}(\xi)-\Theta_{\lambda}(\zeta)|^2}{|\xi-\zeta|^{n+2\sigma+2}}\,\ud\xi \ud\zeta\nonumber\\
&\quad + \iint_{(U\times U)\setminus (B_1^+\times B_1^+)}\frac{(\xi_n-\zeta_n)\left(\sum_{i=1}^{n-1}\alpha_i\xi_i^2-\sum_{i=1}^{n-1}\alpha_i\zeta_i^2\right)|\theta_{\lambda}(\xi)-\theta_{\lambda}(\zeta)|^2}{|\xi-\zeta|^{n+2\sigma+2}}\,\ud\xi \ud\zeta\nonumber\\
& = \iint_{\R^n_+\times\R^n_+}\frac{(\xi_n-\zeta_n)\left(\sum_{i=1}^{n-1}\alpha_i\xi_i^2-\sum_{i=1}^{n-1}\alpha_i\zeta_i^2\right)|\Theta_{\lambda}(\xi)-\Theta_{\lambda}(\zeta)|^2}{|\xi-\zeta|^{n+2\sigma+2}}\,\ud\xi \ud\zeta\nonumber\\
&\quad - \iint_{(\R^n_+\times\R^n_+)\setminus(B_1^+\times B_1^+)}\frac{(\xi_n-\zeta_n)\left(\sum_{i=1}^{n-1}\alpha_i\xi_i^2-\sum_{i=1}^{n-1}\alpha_i\zeta_i^2\right)|\Theta_{\lambda}(\xi)-\Theta_{\lambda}(\zeta)|^2}{|\xi-\zeta|^{n+2\sigma+2}}\,\ud\xi \ud\zeta\nonumber\\
&\quad + \iint_{(U\times U)\setminus (B_1^+\times B_1^+)}\frac{(\xi_n-\zeta_n)\left(\sum_{i=1}^{n-1}\alpha_i\xi_i^2-\sum_{i=1}^{n-1}\alpha_i\zeta_i^2\right)|\theta_{\lambda}(\xi)-\theta_{\lambda}(\zeta)|^2}{|\xi-\zeta|^{n+2\sigma+2}}\,\ud\xi \ud\zeta. \label{eq:term2-1}
\end{align}
\endgroup
Using Theorem \ref{thm:symmetry}, we have
\begingroup
\allowdisplaybreaks
\begin{align}
&\iint_{\R^n_+\times\R^n_+}\frac{(\xi_n-\zeta_n)\left(\sum_{i=1}^{n-1}\alpha_i\xi_i^2-\sum_{i=1}^{n-1}\alpha_i\zeta_i^2\right)|\Theta_{\lambda}(\xi)-\Theta_{\lambda}(\zeta)|^2}{|\xi-\zeta|^{n+2\sigma+2}}\,\ud\xi \ud\zeta\nonumber\\
&=\frac{\sum_{i=1}^{n-1}\alpha_i}{n-1} \iint_{\R^n_+\times\R^n_+}\frac{(\xi_n-\zeta_n)\left(|\xi'|^2- |\zeta'|^2\right)|\Theta_{\lambda}(\xi)-\Theta_{\lambda}(\zeta)|^2}{|\xi-\zeta|^{n+2\sigma+2}}\,\ud\xi \ud\zeta\nonumber\\
&=\frac{H \Gamma_0}{\lambda},\label{eq:term2-2}
\end{align}
\endgroup
where $\Gamma_0$ is given in \eqref{eq:integralsign}, and 
\[
H=\frac{1}{n-1}\sum_{i=1}^{n-1}\alpha_i\quad\mbox{is the mean curvature.}
\]
Since
\[
\frac{ |\xi_n-\zeta_n|\left|\xi_i^2-\zeta_i^2\right|}{|\xi-\zeta|^{n+2\sigma+2}}\le \frac{(|\xi_i|+|\zeta_i|)}{2|\xi-\zeta|^{n+2\sigma}},
\]
we have 
\begingroup
\allowdisplaybreaks
\begin{align}
&\left|\lambda\iint_{(\R^n_+\times\R^n_+)\setminus(B_1^+\times B_1^+)}\frac{(\xi_n-\zeta_n)\left(\sum_{i=1}^{n-1}\alpha_i\xi_i^2-\sum_{i=1}^{n-1}\alpha_i\zeta_i^2\right)|\Theta_{\lambda}(\xi)-\Theta_{\lambda}(\zeta)|^2}{|\xi-\zeta|^{n+2\sigma+2}}\,\ud\xi \ud\zeta\right|\nonumber\\
&= \left|\iint_{(\R^n_+\times\R^n_+)\setminus(B_\lambda^+\times B_\lambda^+)}\frac{(\xi_n-\zeta_n)\left(\sum_{i=1}^{n-1}\alpha_i\xi_i^2-\sum_{i=1}^{n-1}\alpha_i\zeta_i^2\right)|\Theta(\xi)-\Theta(\zeta)|^2}{|\xi-\zeta|^{n+2\sigma+2}}\,\ud\xi \ud\zeta\right|\nonumber\\
&\le \frac{\va_0}{2}\sum_{i=1}^{n-1} \left|\iint_{(\R^n_+\times\R^n_+)\setminus(B_\lambda^+\times B_\lambda^+)}\frac{\left( |\xi_i|+|\zeta_i|\right)|\Theta(\xi)-\Theta(\zeta)|^2}{|\xi-\zeta|^{n+2\sigma}}\,\ud\xi \ud\zeta\right|\nonumber\\
&\le C \va_0\lambda^{1-2\sigma},\label{eq:term2-3}
\end{align}
\endgroup
where we used \eqref{eq:integralsignest} in the last inequality, and 
\begingroup
\allowdisplaybreaks
\begin{align}
&\left|\lambda\iint_{(U\times U)\setminus (B_1^+\times B_1^+)}\frac{(\xi_n-\zeta_n)\left(\sum_{i=1}^{n-1}\alpha_i\xi_i^2-\sum_{i=1}^{n-1}\alpha_i\zeta_i^2\right)|\theta_{\lambda}(\xi)-\theta_{\lambda}(\zeta)|^2}{|\xi-\zeta|^{n+2\sigma+2}}\,\ud\xi \ud\zeta\right|\nonumber\\
&\le \frac{\va_0}{2}\sum_{i=1}^{n-1} \iint_{(\R^n_+\times\R^n_+)\setminus(B_\lambda^+\times B_\lambda^+)}\frac{(|\xi_i|+|\zeta_i|)|\eta(\lambda^{-1}\xi)\Theta(\xi)-\eta(\lambda^{-1}\zeta)\Theta(\zeta)|^2}{|\xi-\zeta|^{n+2\sigma}}\,\ud\xi \ud\zeta\nonumber\\
&= \frac{\va_0}{2}\sum_{i=1}^{n-1} \iint_{(\R^n_+\times\R^n_+)\setminus(B_\lambda^+\times B_\lambda^+)}\frac{|\zeta_i||\eta(\lambda^{-1}\xi)\Theta(\xi)-\eta(\lambda^{-1}\zeta)\Theta(\zeta)|^2}{|\xi-\zeta|^{n+2\sigma}}\,\ud\xi \ud\zeta\nonumber\\
&\le \frac{\va_0}{2}\sum_{i=1}^{n-1} \iint_{(\R^n_+\times\R^n_+)\setminus(B_\lambda^+\times B_\lambda^+)}\frac{2|\zeta_i| |\Theta(\xi)-\Theta(\zeta)|^2}{|\xi-\zeta|^{n+2\sigma}}\,\ud\xi \ud\zeta\nonumber\\
&\quad + \frac{\va_0}{2}\sum_{i=1}^{n-1} \iint_{(\R^n_+\times\R^n_+)\setminus(B_\lambda^+\times B_\lambda^+)}\frac{2|\zeta_i| |\Theta(\zeta)|^2 |\eta(\lambda^{-1}\xi)-\eta(\lambda^{-1}\zeta)|^2}{|\xi-\zeta|^{n+2\sigma}}\,\ud\xi \ud\zeta,\label{eq:term2-4}
\end{align}
\endgroup
where we used \eqref{eq:CWcutoff} in the last inequality. Since
\begingroup
\allowdisplaybreaks
\begin{align*}
&\int_{\R^n_+ \setminus B_\lambda^+} |\zeta_i| |\Theta(\zeta)|^2\,\ud\zeta \int_{\R^n}\frac{ |\eta(\lambda^{-1}\xi)-\eta(\lambda^{-1}\zeta)|^2}{|\xi-\zeta|^{n+2\sigma}}\, \ud\xi\\
&=\int_{B_{4\lambda}^+ \setminus B_\lambda^+} |\zeta_i| |\Theta(\zeta)|^2\,\ud\zeta \int_{\R^n}\frac{ |\eta(\lambda^{-1}\xi)-\eta(\lambda^{-1}\zeta)|^2}{|\xi-\zeta|^{n+2\sigma}}\, \ud\xi\\
&\quad+ \int_{\R^n_+ \setminus B_{4\lambda}^+} |\zeta_i| |\Theta(\zeta)|^2\,\ud\zeta \int_{B_{3\lambda}^+}\frac{1}{|\xi-\zeta|^{n+2\sigma}}\, \ud\xi\\
& \le C\lambda^{3-n-2\sigma}+C \lambda^n \int_{\R^n_+ \setminus B_{4\lambda}^+} |\zeta|^{1-n-2\sigma} |\Theta(\zeta)|^2\,\ud\zeta \\
& \le C\lambda^{3-n-2\sigma}\\
&\le C\lambda^{1-2\sigma},
\end{align*}
\endgroup
and
\begingroup
\allowdisplaybreaks
\begin{align*}
&\int_{B_\lambda^+} |\zeta_i| |\Theta(\zeta)|^2\,\ud\zeta \int_{\R^n_+ \setminus B_\lambda^+}\frac{ |\eta(\lambda^{-1}\xi)-\eta(\lambda^{-1}\zeta)|^2}{|\xi-\zeta|^{n+2\sigma}}\, \ud\xi\\
&=\int_{B_\lambda^+} |\zeta_i| |\Theta(\zeta)|^2\,\ud\zeta \int_{\R^n_+ \setminus B_{2\lambda}^+}\frac{ |\eta(\lambda^{-1}\xi)-\eta(\lambda^{-1}\zeta)|^2}{|\xi-\zeta|^{n+2\sigma}}\, \ud\xi\\
&\le C \lambda^{-2\sigma} \int_{B_\lambda^+} |\zeta_i| |\Theta(\zeta)|^2\,\ud\zeta \\
&\le C \lambda^{-2\sigma} (1+\lambda^{3-n})\\
&\le C\lambda^{1-2\sigma},
\end{align*}
\endgroup
we obtain from \eqref{eq:term2-4} that
\begin{align}
&\left|\lambda\iint_{(U\times U)\setminus (B_1^+\times B_1^+)}\frac{(\xi_n-\zeta_n)\left(\sum_{i=1}^{n-1}\alpha_i\xi_i^2-\sum_{i=1}^{n-1}\alpha_i\zeta_i^2\right)|\theta_{\lambda}(\xi)-\theta_{\lambda}(\zeta)|^2}{|\xi-\zeta|^{n+2\sigma+2}}\,\ud\xi \ud\zeta\right|\nonumber\\
&\le C\va_0\lambda^{1-2\sigma}.\label{eq:term2-5}
\end{align}
Therefore, it follows from \eqref{eq:term2-1}, \eqref{eq:term2-2}, \eqref{eq:term2-3} and \eqref{eq:term2-5} that we obtain the estimate for the second term in the right hand side of \eqref{eq:expansionofI}:
\begin{align}
&\left|\iint_{U\times U}\frac{(\xi_n-\zeta_n)\left(\sum_{i=1}^{n-1}\alpha_i\xi_i^2-\sum_{i=1}^{n-1}\alpha_i\zeta_i^2\right)|\theta_{\lambda}(\xi)-\theta_{\lambda}(\zeta)|^2}{|\xi-\zeta|^{n+2\sigma+2}}\,\ud\xi \ud\zeta-\frac{H \Gamma_0}{\lambda}\right|\nonumber\\
&\le C \varepsilon_0 \lambda^{-2\sigma}\label{eq:leadingterm}.
\end{align} 
Combining \eqref{eq:expansionofI}, \eqref{eq:limitestimate} and \eqref{eq:leadingterm}, we have 
\begin{align}
I_{n,\sigma,\Omega}[v_{\lambda}]&\leq\iint_{U \times U}\frac{|\theta_{\lambda}(\xi)-\theta_{\lambda}(\zeta)|^2}{|\xi-\zeta|^{n+2\sigma}} \ud\xi\ud\zeta-\frac{(n+2\sigma)H \Gamma_0}{2\lambda}+\frac{C\varepsilon_0 }{\lambda^{2\sigma}}\label{eq:energynumerator}.
\end{align}
From the proof of Theorem 1.3 in \cite{FJX}, we have
\begin{align}
\iint_{U\times U}\frac{|\theta_{\lambda}(\xi)-\theta_{\lambda}(\zeta)|^2}{|\xi-\zeta|^{n+2\sigma}}\,\ud\xi\ud\zeta& \le S_{n,\sigma}(\mathbb{R}_+^n)-c\lambda^{-2\sigma}+C\lambda^{-n-2\sigma+2}\label{eq:energynumerator2}\\
\int_{\Omega}|v_{\lambda}|^{\frac{2n}{n-2\sigma}}\,\ud x \ge \int_{B_{4}}|\theta_{\lambda}|^{\frac{2n}{n-2\sigma}}\,\ud \xi &\ge 1-c\lambda^{-\frac{n(n+2\sigma-2)}{n-2\sigma}},\label{eq:energydenominator}
\end{align}
where $c$ and $C$ are positive constants depending only on $n$ and $\sigma$. Hence, we have from \eqref{eq:energynumerator}, \eqref{eq:energynumerator2} and \eqref{eq:energydenominator} that 
\begin{align}
&\frac{I_{n,\sigma,\Omega}[v_{\lambda}]}{\left(\int_{\Omega}|v_{\lambda}|^{\frac{2n}{n-2\sigma}}dx\right)^{\frac{n-2\sigma}{n}}}\nonumber\\
&\leq\left(1+C\lambda^{-\frac{n(n+2\sigma-2)}{n-2\sigma}}\right)\cdot \left(S_{n,\sigma}(\mathbb{R}_+^n)-\frac{(n+2\sigma)H \Gamma_0}{2\lambda}-(c-C\varepsilon_0)\lambda^{-2\sigma}+C\lambda^{-n-2\sigma+2}\right)\nonumber\\
&\le S_{n,\sigma}(\mathbb{R}_+^n)-\frac{(n+2\sigma)H \Gamma_0}{2\lambda}-(c-C\varepsilon_0)\lambda^{-2\sigma}+C\lambda^{-n+1}\label{eq:finalexp1}.
\end{align}
Without knowing the sign of $\Gamma_0$, we use the crude estimate that $|H|\le \va_0$. Therefore,
\[
\frac{I_{n,\sigma,\Omega}[v_{\lambda}]}{\left(\int_{\Omega}|v_{\lambda}|^{\frac{2n}{n-2\sigma}}dx\right)^{\frac{n-2\sigma}{n}}}\le S_{n,\sigma}(\mathbb{R}_+^n)-\frac{(n+2\sigma) \Gamma_0\va_0}{2\lambda}-(c-C\varepsilon_0)\lambda^{-2\sigma}+C\lambda^{-n+1}
\]
By choosing  $\lambda$ large and then choosing $\varepsilon_0$ small, we obtain $ S_{n,\sigma}(\Omega)< S_{n,\sigma}(\mathbb{R}_+^n) $. 
\end{proof}

\begin{proof}[Proof of Theorem \ref{thm:generaldomains}]
As mentioned earlier, we can assume $a=0$ and there exists $\delta_0>0$ such that $\partial\Omega\cap B_{\delta_0}$ can be represented by \eqref{eq:graphfunction} after a necessary coordinate rotation.  Since  the sharp constant $ S_{n,\sigma}(\Omega)$ does not change under dilations, we can have a dilation of $ \Omega $ with a sufficiently large number $ \mu $. The domain after dilation is denoted as $$ \Omega_{\mu}:=\{\mu x: x\in \Omega\}.$$ Then the boundary $\partial\Omega_\mu\cap B_{\mu\delta_0}$ is presented by 
$$
x_n=h_{\mu}(x'):=\frac{1}{2\mu}\sum_{i=1}^{n-1}\alpha_i x_{i}^2+\frac{1}{\mu} g\left(\frac{1}{\mu}x'\right)|x'|^2.
$$ 
Choose $\mu$ large so that  $(B_8\cap \mathscr{R}_\mu)\subset \Omega_{\mu}$, where $\mathscr{R}_\mu := \{x\in\R^n: x_n>h_{\mu}(x')\}$. 

Let $\Phi_{\mu}:\Omega_{\mu} \rightarrow\mathbb{R}^n$ defined as $\xi=\Phi_{\mu}(x)=(x_1,\cdots, x_{n-1}, x_n-h_{\mu}(x'))$. Let $ \Theta(x) $ be a minimizer of $ S_{n,\sigma}(\mathbb{R}_+^n) $ that is radially symmetric in the first $n-1$ variables as before, and $\Theta_{\lambda}(x)=\lambda^{\frac{n-2\sigma}{2}}\Theta({\lambda}x)$ for $\lambda>0$. 
Let $ \eta $ be a cut off function such that $ \eta\in C^{1}(\mathbb{R}^n) $, $ \eta\equiv 1 $ in $ B_{2} $, $ 0\leq \eta\leq 1 $ in $ B_{3} $ and $ \eta\equiv 0 $ in $ B_{3}^{c} $.  Let $ \theta_{\lambda}(x)=(\eta\Theta_{\lambda})(x) $ and $ v_{\lambda}=\theta_{\lambda}\circ\Phi_{\mu}(x) $. Then we have 
\begingroup
\allowdisplaybreaks     
\begin{align}
I_{n,\sigma,\Omega_{\mu}}[v_{\lambda}]&=\iint_{\Omega_{\mu}\times\Omega_{\mu}}\frac{|v_{\lambda}(x)-v_{\lambda}(y)|^2}{|x-y|^{n+2\sigma}}\,\ud x \ud y\nonumber\\
&=\iint_{(\Omega_{\mu}\cap B_8)\times (\Omega_{\mu}\cap B_8)}\frac{|v_{\lambda}(x)-v_{\lambda}(y)|^2}{|x-y|^{n+2\sigma}}\,\ud x \ud y\nonumber\\
&\quad+2\int_{\Omega_{\mu}\cap B_8}|v_{\lambda}(x)|^2\int_{\Omega_{\mu}\backslash  B_8}\frac{\ud y}{|x-y|^{n+2\sigma}}\,\ud x\nonumber.
\end{align}
\endgroup
Since
\begingroup
\allowdisplaybreaks
\begin{align*}
&\int_{\Omega_{\mu}\cap B_8}|v_{\lambda}(x)|^2\int_{\Omega_{\mu}\backslash  B_8}\frac{\ud y}{|x-y|^{n+2\sigma}}\,\ud x \\
&=2\int_{\Omega_{\mu} \cap B_{4}}|v_{\lambda}(x)|^2\int_{\Omega_{\mu}\backslash B_8}\frac{\ud y}{|x-y|^{n+2\sigma}}\,\ud x\\
&\leq \frac{C(n,\sigma)}{\lambda^{2\sigma}}\int_{\mathbb{R}_+^n}|\Theta(x)|^2dx \nonumber\\
&\leq \frac{C(n,\sigma)}{\lambda^{2\sigma}},
\end{align*}
\endgroup
where we used Proposition \ref{lem:minimizerbound} in the last inequality, it follows from \eqref{eq:energynumerator} and \eqref{eq:energynumerator2} that 
\[
I_{n,\sigma,\Omega_{\mu}}[v_{\lambda}]\le S_{n,\sigma}(\mathbb{R}_+^n) -\frac{(n+2\sigma)\Gamma_0 H}{2\mu\lambda}+C\lambda^{-2\sigma}.
\]
This together with \eqref{eq:energydenominator} shows that 
\[
S_{n,\sigma}(\Omega)\le S_{n,\sigma}(\mathbb{R}_+^n)-\frac{(n+2\sigma) \Gamma_0 H}{2\mu\lambda}+C\lambda^{-2\sigma}.
\]
\end{proof}


\appendix

\section{Proof of Proposition \ref{prop:integrability}}\label{sec:appendix}
We first prove \eqref{eq:integrability2}.  

\begin{proof}[Proof of \eqref{eq:integrability2}]
We suppose $\lambda>100$ is very large. Note that
\[
\iint_{B_\lambda^+\times B_\lambda^+}\frac{ (|\xi'|^\gamma+|\zeta'|^\gamma) |\Theta(\xi)-\Theta(\zeta)|^2}{|\xi-\zeta|^{n+2\sigma}}\ud\xi \ud\zeta=2\iint_{B_\lambda^+\times B_\lambda^+}\frac{ |\xi'|^\gamma |\Theta(\xi)-\Theta(\zeta)|^2}{|\xi-\zeta|^{n+2\sigma}}\ud\xi \ud\zeta.
\]

First, it is clear that
\begin{align}\label{eq:integrability2-1}
\int_{B_{10}^+} |\xi'|^\gamma\,\ud \xi \int_{B_\lambda^+}\frac{|\Theta(\xi)-\Theta(\zeta)|^2}{|\xi-\zeta|^{n+2\sigma}}\,\ud\zeta\le 10^\gamma \iint_{\mathbb{R}_+^n\times\mathbb{R}_+^n}\frac{  |\Theta(\xi)-\Theta(\zeta)|^2}{|\xi-\zeta|^{n+2\sigma}}\ud\xi \ud\zeta<\infty.
\end{align}

Secondly, using the Cauchy-Schwarz inequality, we have
\begin{align}
&\int_{B_\lambda^+\setminus B_{10}^+} |\xi'|^\gamma\,\ud \xi \int_{\{\zeta\in B_\lambda^+:\, |\zeta-\xi|\ge 1\}}\frac{|\Theta(\xi)-\Theta(\zeta)|^2}{|\xi-\zeta|^{n+2\sigma}}\,\ud\zeta\nonumber\\
&\le C \int_{B_\lambda^+\setminus B_{10}^+} |\xi'|^\gamma \Theta(\xi)^2\,\ud \xi + \int_{B_\lambda^+\setminus B_{10}^+} |\xi'|^\gamma\,\ud \xi \int_{\{\zeta\in B_\lambda^+:\, |\zeta-\xi|\ge 1\}}\frac{2\Theta(\zeta)^2}{|\xi-\zeta|^{n+2\sigma}}\,\ud\zeta\label{eq:integrability2-2}.
\end{align}
We have
\begin{align}
\int_{B_\lambda^+\setminus B_{10}^+} |\xi'|^\gamma\,\ud \xi \int_{\{\zeta\in B_\lambda^+:\, |\zeta-\xi|\ge 1,\ |\zeta|\le |\xi|/2\}}\frac{2\Theta(\zeta)^2}{|\xi-\zeta|^{n+2\sigma}}\,\ud\zeta  & \le C \int_{B_\lambda^+\setminus B_{10}^+} \frac{|\xi'|^\gamma}{|\xi|^{n+2\sigma}}\,\ud \xi\nonumber\\
&\le C(1+\lambda^{\gamma-2\sigma}),\nonumber
\end{align}
and
\begin{align}
\int_{B_\lambda^+\setminus B_{10}^+} |\xi'|^\gamma\,\ud \xi \int_{\{\zeta\in B_\lambda^+:\, |\zeta-\xi|\ge 1,\ |\zeta|\ge |\xi|/2\}}\frac{2\Theta(\zeta)^2}{|\xi-\zeta|^{n+2\sigma}}\,\ud\zeta  &\le C \int_{B_\lambda^+\setminus B_{10}^+} \frac{|\xi'|^\gamma} {|\xi|^{2n-2}}\,\ud \xi \nonumber\\
&<C(1+\lambda^{\gamma-2\sigma}),\nonumber
\end{align}
where we used \eqref{eq:minimizerbound} and $n\ge 4>2+2\sigma$.  Hence, it follows from \eqref{eq:integrability2-2} that 
\begin{align}
&\int_{B_\lambda^+ \setminus B_{10}^+} |\xi'|^\gamma\,\ud \xi \int_{\{\zeta\in B_\lambda^+:\, |\zeta-\xi|\ge 1\}}\frac{|\Theta(\xi)-\Theta(\zeta)|^2}{|\xi-\zeta|^{n+2\sigma}}\,\ud\zeta\le C(1+\lambda^{\gamma-2\sigma}).\label{eq:integrability2-3}
\end{align}

Finally, if we denote $\widetilde \Theta(\xi)=\Theta(\xi) \xi_{n}^{1-2\sigma}$, then
\begingroup
\allowdisplaybreaks
\begin{align}
&\int_{B_\lambda^+\setminus B_{10}^+} |\xi'|^\gamma\,\ud \xi \int_{\{\zeta\in B_\lambda^+:\, |\zeta-\xi|\le 1\}}\frac{|\Theta(\xi)-\Theta(\zeta)|^2}{|\xi-\zeta|^{n+2\sigma}}\,\ud\zeta\nonumber\\
&= \int_{B_\lambda^+ \setminus B_{10}^+} |\xi'|^\gamma\,\ud \xi \int_{\{\zeta\in B_\lambda^+:\, |\zeta-\xi|\le 1\}}\frac{|\widetilde\Theta(\xi) \xi_{n}^{2\sigma-1}-\widetilde\Theta(\zeta)\zeta_n^{2\sigma-1}|^2}{|\xi-\zeta|^{n+2\sigma}}\,\ud\zeta\nonumber\\
&\le \int_{B_\lambda^+ \setminus B_{10}^+} |\xi'|^\gamma\widetilde\Theta(\xi)^2 \,\ud \xi \int_{\{\zeta\in B_\lambda^+:\, |\zeta-\xi|\le 1\}}\frac{| \xi_{n}^{2\sigma-1}-\zeta_n^{2\sigma-1}|^2}{|\xi-\zeta|^{n+2\sigma}}\,\ud\zeta\nonumber\\
&\quad+ \int_{B_\lambda^+ \setminus B_{10}^+} |\xi'|^\gamma\,\ud \xi \int_{\{\zeta\in B_\lambda^+:\, |\zeta-\xi|\le 1\}}\frac{|\widetilde\Theta(\xi)-\widetilde\Theta(\zeta)|^2 \zeta_n^{4\sigma-2}}{|\xi-\zeta|^{n+2\sigma}}\,\ud\zeta.\label{eq:integrability2-4}
\end{align}
\endgroup
Note that if $\xi_n\ge 3/2$, then 
\begin{align}
\int_{\{\zeta\in B_\lambda^+:\, |\zeta-\xi|\le 1\}}\frac{| \xi_{n}^{2\sigma-1}-\zeta_n^{2\sigma-1}|^2}{|\xi-\zeta|^{n+2\sigma}}\,\ud\zeta\le C \xi_n^{2\sigma-2}<C. \nonumber
\end{align}
Now let us consider $\xi_n< 3/2$. Then
\begin{align}
\int_{\{\zeta\in B_\lambda^+:\, |\zeta-\xi|\le \frac{\xi_n}{2}\}}\frac{| \xi_{n}^{2\sigma-1}-\zeta_n^{2\sigma-1}|^2}{|\xi-\zeta|^{n+2\sigma}}\,\ud\zeta &\le C \xi_n^{2\sigma-2}, \nonumber\\
\int_{\{\zeta\in B_\lambda^+:\, \frac{\xi_n}{2}<|\zeta-\xi|\le 1,\ \zeta_n<2\xi_n \}}\frac{| \xi_{n}^{2\sigma-1}-\zeta_n^{2\sigma-1}|^2}{|\xi-\zeta|^{n+2\sigma}}\,\ud\zeta &\le C \xi_n^{4\sigma-2}\cdot \xi_n^{-2\sigma}=C \xi_n^{2\sigma-2}, \nonumber\\
\int_{\{\zeta\in B_\lambda^+:\, \frac{\xi_n}{2}<|\zeta-\xi|\le 1,\ \zeta_n\ge 2\xi_n \}}\frac{| \xi_{n}^{2\sigma-1}-\zeta_n^{2\sigma-1}|^2}{|\xi-\zeta|^{n+2\sigma}}\,\ud\zeta &\le C\int_{\{\zeta\in B_\lambda^+:\, |\zeta|\le 1,\ \zeta_n\ge \xi_n \}}\frac{\zeta_n^{4\sigma-2}}{|\zeta|^{n+2\sigma}}\,\ud\zeta \le C \xi_n^{2\sigma-2}.\nonumber
\end{align}
Therefore,
\begin{align*}
 \int_{B_\lambda^+ \setminus B_{10}^+} |\xi'|^\gamma\widetilde\Theta(\xi)^2 \,\ud \xi \int_{\{\zeta\in B_\lambda^+:\, |\zeta-\xi|\le 1\}}\frac{| \xi_{n}^{2\sigma-1}-\zeta_n^{2\sigma-1}|^2}{|\xi-\zeta|^{n+2\sigma}}\,\ud\zeta & \le   C\int_{B_\lambda^+\setminus B_{10}^+} |\xi'|^\gamma \xi_n^{2\sigma-2}\widetilde\Theta(\xi)^2 \,\ud \xi\\
 &\le C(1+\lambda^{\gamma-2\sigma}), 
\end{align*}
where we used \eqref{eq:minimizerbound} and $n\ge 4$ in the last inequality. Furthermore, using \eqref{eq:gradientbound}, we have
\begin{align*}
\int_{B_\lambda^+\setminus B_{10}^+} |\xi'|^\gamma\,\ud \xi \int_{\{\zeta\in B_\lambda^+:\, |\zeta-\xi|\le 1\}}\frac{|\widetilde\Theta(\xi)-\widetilde\Theta(\zeta)|^2 \zeta_n^{4\sigma-2}}{|\xi-\zeta|^{n+2\sigma}}\,\ud\zeta &\le C \int_{B_\lambda^+\setminus B_{10}^+} |\xi'|^\gamma \frac{1}{|\xi|^{2n}}\,\ud \xi\\
 &\le C(1+\lambda^{\gamma-2\sigma}). 
\end{align*}
Hence, it follows from \eqref{eq:integrability2-4} that
\[
\int_{B_\lambda^+ \setminus B_{10}^+} |\xi'|^\gamma\,\ud \xi \int_{\{\zeta\in B_\lambda^+:\, |\zeta-\xi|\le 1\}}\frac{|\Theta(\xi)-\Theta(\zeta)|^2}{|\xi-\zeta|^{n+2\sigma}}\,\ud\zeta \le C(1+\lambda^{\gamma-2\sigma}).
\]

This, together with \eqref{eq:integrability2-1} and \eqref{eq:integrability2-3}, proves \eqref{eq:integrability2}.
\end{proof}

Next, we prove \eqref{eq:integralsignest}.
\begin{proof}[Proof of \eqref{eq:integralsignest}]
We suppose $\lambda>100$ is very large. Note that
\begin{align}
&\iint_{(\mathbb{R}_+^n\times\mathbb{R}_+^n)\setminus (B_\lambda^+\times B_\lambda^+)}\frac{ (|\xi'|^\gamma+|\zeta'|^\gamma) |\Theta(\xi)-\Theta(\zeta)|^2}{|\xi-\zeta|^{n+2\sigma}}\ud\xi \ud\zeta \nonumber \\
&\le 4 \iint_{(\mathbb{R}_+^n\setminus B_\lambda^+)\times\mathbb{R}_+^n}\frac{ |\xi'|^\gamma |\Theta(\xi)-\Theta(\zeta)|^2}{|\xi-\zeta|^{n+2\sigma}}\ud\xi \ud\zeta. \label{eq:integrability2-1-2}
\end{align}
Again, using the Cauchy-Schwarz inequality, we have
\begin{align}
&\int_{\R^n_+\setminus B_{\lda}^+} |\xi'|^\gamma\,\ud \xi \int_{\{\zeta\in\R^n_+:\, |\zeta-\xi|\ge 1\}}\frac{|\Theta(\xi)-\Theta(\zeta)|^2}{|\xi-\zeta|^{n+2\sigma}}\,\ud\zeta\nonumber\\
&\le C \int_{\R^n_+\setminus B_{\lda}^+} |\xi'|^\gamma \Theta(\xi)^2\,\ud \xi + \int_{\R^n\setminus B_{\lda}^+} |\xi'|^\gamma\,\ud \xi \int_{\{\zeta\in\R^n_+:\, |\zeta-\xi|\ge 1\}}\frac{2\Theta(\zeta)^2}{|\xi-\zeta|^{n+2\sigma}}\,\ud\zeta\label{eq:integrability2-2-2}.
\end{align}
Since $\gamma<2\sigma$, we have
\begin{align}
\int_{\R^n_+\setminus B_{\lda}^+} |\xi'|^\gamma\,\ud \xi \int_{\{\zeta\in\R^n_+:\, |\zeta-\xi|\ge 1,\ |\zeta|\le |\xi|/2\}}\frac{2\Theta(\zeta)^2}{|\xi-\zeta|^{n+2\sigma}}\,\ud\zeta  \le C \int_{\mathbb{R}_+^n\setminus B_{\lambda}^+} \frac{|\xi'|^\gamma}{|\xi|^{n+2\sigma}}\,\ud \xi \le C\lambda^{\gamma-2\sigma},\nonumber
\end{align}
and
\begin{align}
\int_{\R^n_+\setminus B_{\lambda}^+} |\xi'|^\gamma\,\ud \xi \int_{\{\zeta\in\R^n_+:\, |\zeta-\xi|\ge 1,\ |\zeta|\ge |\xi|/2\}}\frac{2\Theta(\zeta)^2}{|\xi-\zeta|^{n+2\sigma}}\,\ud\zeta  \le C \int_{\R^n_+\setminus B_{\lambda}^+} \frac{|\xi'|^\gamma} {|\xi|^{2n-2}}\,\ud \xi\le C\lambda^{\gamma-2\sigma},\nonumber
\end{align}
where we used \eqref{eq:minimizerbound} and $n\ge 4>2+2\sigma$.  Hence, it follows from \eqref{eq:integrability2-2-2} that 
\begin{align}
&\int_{\R^n_+\setminus B_{\lda}^+} |\xi'|^\gamma\,\ud \xi \int_{\{\zeta\in\R^n_+:\, |\zeta-\xi|\ge 1\}}\frac{|\Theta(\xi)-\Theta(\zeta)|^2}{|\xi-\zeta|^{n+2\sigma}}\,\ud\zeta \le C\lambda^{\gamma-2\sigma}.\label{eq:integrability2-3-2}
\end{align}

Finally, if we denote $\widetilde \Theta(\xi)=\Theta(\xi) \xi_{n}^{1-2\sigma}$, then
\begingroup
\allowdisplaybreaks
\begin{align}
&\int_{\R^n_+\setminus B_{\lda}^+} |\xi'|^\gamma\,\ud \xi \int_{\{\zeta\in\R^n_+:\, |\zeta-\xi|\le 1\}}\frac{|\Theta(\xi)-\Theta(\zeta)|^2}{|\xi-\zeta|^{n+2\sigma}}\,\ud\zeta\nonumber\\
&= \int_{\R^n_+\setminus B_{\lda}^+} |\xi'|^\gamma\,\ud \xi \int_{\{\zeta\in\R^n_+:\, |\zeta-\xi|\le 1\}}\frac{|\widetilde\Theta(\xi) \xi_{n}^{2\sigma-1}-\widetilde\Theta(\zeta)\zeta_n^{2\sigma-1}|^2}{|\xi-\zeta|^{n+2\sigma}}\,\ud\zeta\nonumber\\
&\le \int_{\R^n_+\setminus B_{\lda}^+} |\xi'|^\gamma\widetilde\Theta(\xi)^2 \,\ud \xi \int_{\{\zeta\in\R^n_+:\, |\zeta-\xi|\le 1\}}\frac{| \xi_{n}^{2\sigma-1}-\zeta_n^{2\sigma-1}|^2}{|\xi-\zeta|^{n+2\sigma}}\,\ud\zeta\nonumber\\
&\quad+ \int_{\R^n_+\setminus B_{\lda}^+} |\xi'|^\gamma\,\ud \xi \int_{\{\zeta\in\R^n_+:\, |\zeta-\xi|\le 1\}}\frac{|\widetilde\Theta(\xi)-\widetilde\Theta(\zeta)|^2 \zeta_n^{4\sigma-2}}{|\xi-\zeta|^{n+2\sigma}}\,\ud\zeta.\label{eq:integrability2-4-2}
\end{align}
\endgroup
Note that if $\xi_n\ge 3/2$, then 
\begin{align}
\int_{\{\zeta\in\R^n_+:\, |\zeta-\xi|\le 1\}}\frac{| \xi_{n}^{2\sigma-1}-\zeta_n^{2\sigma-1}|^2}{|\xi-\zeta|^{n+2\sigma}}\,\ud\zeta\le C \xi_n^{2\sigma-2}<C. \nonumber
\end{align}
Now let us consider $\xi_n< 3/2$. Then
\begingroup
\allowdisplaybreaks
\begin{align}
\int_{\{\zeta\in\R^n_+:\, |\zeta-\xi|\le \frac{\xi_n}{2}\}}\frac{| \xi_{n}^{2\sigma-1}-\zeta_n^{2\sigma-1}|^2}{|\xi-\zeta|^{n+2\sigma}}\,\ud\zeta &\le C \xi_n^{2\sigma-2}, \nonumber\\
\int_{\{\zeta\in\R^n_+:\, \frac{\xi_n}{2}<|\zeta-\xi|\le 1,\ \zeta_n<2\xi_n \}}\frac{| \xi_{n}^{2\sigma-1}-\zeta_n^{2\sigma-1}|^2}{|\xi-\zeta|^{n+2\sigma}}\,\ud\zeta &\le C \xi_n^{4\sigma-2}\cdot \xi_n^{-2\sigma}=C \xi_n^{2\sigma-2}, \nonumber\\
\int_{\{\zeta\in\R^n_+:\, \frac{\xi_n}{2}<|\zeta-\xi|\le 1,\ \zeta_n\ge 2\xi_n \}}\frac{| \xi_{n}^{2\sigma-1}-\zeta_n^{2\sigma-1}|^2}{|\xi-\zeta|^{n+2\sigma}}\,\ud\zeta &\le C\int_{\{\zeta\in\R^n_+:\, |\zeta|\le 1,\ \zeta_n\ge \xi_n \}}\frac{\zeta_n^{4\sigma-2}}{|\zeta|^{n+2\sigma}}\,\ud\zeta \le C \xi_n^{2\sigma-2}.\nonumber
\end{align}
\endgroup
Therefore,
\begin{align*}
 \int_{\R^n_+\setminus B_{\lda}^+} |\xi'|^\gamma\widetilde\Theta(\xi)^2 \,\ud \xi \int_{\{\zeta\in\R^n_+:\, |\zeta-\xi|\le 1\}}\frac{| \xi_{n}^{2\sigma-1}-\zeta_n^{2\sigma-1}|^2}{|\xi-\zeta|^{n+2\sigma}}\,\ud\zeta & \le   C\int_{\R^n_+\setminus B_{\lda}^+} |\xi'|^\gamma \xi_n^{2\sigma-2}\widetilde\Theta(\xi)^2 \,\ud \xi\\
 &\le C\lambda^{\gamma-2\sigma}, 
\end{align*}
where we used \eqref{eq:minimizerbound} and $n\ge 4>4\sigma$ in the last inequality. Furthermore, using \eqref{eq:gradientbound}, we have
\begin{align*}
\int_{\R^n_+\setminus B_{\lda}^+} |\xi'|^\gamma\,\ud \xi \int_{\{\zeta\in\R^n_+:\, |\zeta-\xi|\le 1\}}\frac{|\widetilde\Theta(\xi)-\widetilde\Theta(\zeta)|^2 \zeta_n^{4\sigma-2}}{|\xi-\zeta|^{n+2\sigma}}\,\ud\zeta \le C \int_{\R^n_+\setminus B_{\lda}^+} |\xi'|^\gamma \frac{1}{|\xi|^{2n}}\,\ud \xi\le C\lambda^{\gamma-2\sigma}. 
\end{align*}
Hence, it follows from \eqref{eq:integrability2-4-2} that
\[
\int_{\R^n_+\setminus B_{10}^+} |\xi'|^\gamma\,\ud \xi \int_{\{\zeta\in\R^n_+:\, |\zeta-\xi|\le 1\}}\frac{|\Theta(\xi)-\Theta(\zeta)|^2}{|\xi-\zeta|^{n+2\sigma}}\,\ud\zeta \le C\lambda^{\gamma-2\sigma}.
\]
This, together with \eqref{eq:integrability2-1-2} and \eqref{eq:integrability2-3-2}, proves \eqref{eq:integrability2}.
\end{proof}

\bigskip

\noindent R. L. Frank

\noindent Mathematisches Institut, Ludwig-Maximilans Universit\"at M\"unchen, Theresienstr. 39, 80333 M\"unchen, Germany, Munich Center for Quantum Science and Technology, Schellingstr. 4, 80799, M\"unchen, Germany, and Mathematics 253-37, Caltech, Pasadena, CA 91125, USA\\
Email: \textsf{frank@math.lmu.de, rlfrank@caltech.edu}

\medskip

\noindent T. Jin

\noindent Department of Mathematics, The Hong Kong University of Science and Technology\\
Clear Water Bay, Kowloon, Hong Kong\\
Email: \textsf{tianlingjin@ust.hk}

\medskip

\noindent W. Wang

\noindent School of Mathematical Sciences, Peking University, Beijing, China \\
Email: \textsf{2201110024@stu.pku.edu.cn}

\end{document}